\documentclass[11pt]{amsart}
\usepackage{geometry,amsrefs}                
\usepackage{graphicx}
\usepackage[usenames,dvipsnames]{color}
\usepackage{amssymb}
\usepackage{epstopdf}
\newtheorem{theorem}{Theorem}

\newtheorem{cor}{Corollary}
\newtheorem{proposition}{Proposition}

\newtheorem{df}{Definition}
\title{Normal forms of para-CR hypersurfaces}
\author{Alessandro Ottazzi, Gerd Schmalz}
\address{(Ottazzi) CIRM, Fondazione Bruno Kessler, Via Sommarive, 14 - Povo,
I-38123 Trento, Italy}%
\address{School of Mathematics and Statistics, UNSW, Sydney}
\email{alessandro.ottazzi@gmail.com}
\address{(Schmalz) School of Science and Technology, University of New England, Armidale, NSW 2351
Australia}
\email{gerd@une.edu.au}%

\begin{document}

\thanks{The authors have been partially supported by the ARC Discovery grant DP130103485, by the ARC Discovery grant DP140100531, and by the CIRM Fondazione Bruno Kessler}
\maketitle

\begin{abstract}
We consider hypersurfaces of finite type in a direct product space $\mathbb R^2\times \mathbb R^2$, which are analogues to real hypersurfaces of finite type in $\mathbb C^2$. 
We shall consider separately the cases where such hypersurfaces are regular and singular, in a sense that corresponds to Levi degeneracy in hypersurfaces in  $\mathbb C^2$. For the regular case, we study formal normal forms and prove convergence by following Chern and Moser. The normal form of such an hypersurface, considered as the solution manifold of a 2nd order ODE, gives rise to a normal form of the corresponding 2nd order ODE.  For the degenerate case, we  study normal forms for weighted $\ell$-jets. Furthermore, we study  the automorphisms of finite type hypersurfaces. 
\end{abstract}

\section{Introduction}

It is well known that the $2n+1$-dimensional Heisenberg group, viewed with its CR structure, is the \v{S}ilov boundary of the Siegel domain in $\mathbb C^{n+1}$, namely the quadric $\Im w=|z|^2$, $z\in \mathbb C^n$. It is also well known that the infinitesimal automorphisms of the Heisenberg group with its CR structure form the Lie algebra $\mathfrak{su}(1,{n+1})$. 
On the one hand, the Heisenberg algebra is  the nilpotent component of the Iwasawa decomposition of $\mathfrak{su}(1,{n+1})$. On the other hand, when $n=1$, the three dimensional Heisenberg algebra is also the nilpotent component of the Iwasawa decomposition of $\mathfrak{sl}(3,\mathbb R)$. This fact has a nice geometric interpretation: the three dimensional Heisenberg group with its multicontact structure is diffeomorphic to the hypersurface $y=a+bx$ of $\mathbb R^2_{xy}\times \mathbb R^2_{ab}$ endowed with its para-CR structure \cite{Ottazzi-Schmalz1}. 
For more on
multicontact structures see \cite{CDeKR1, CDeKR2, KorMC, Ottazzi-Hess, Han-Oh-Schmalz, CCDEM, DeOtt}, see also \cite{Yatsui1, Yatsui2}.
The infinitesimal automorphisms of this structure on $y=a+bx$ are given by $\mathfrak{sl}(3,\mathbb R)$.  As one does for CR manifolds, it is then natural to study general hypersurfaces embedded in the product space $\mathbb R^2_{xy}\times \mathbb R^2_{ab}$ and consider their infinitesimal automorphisms and normal forms.

In this order of ideas, we consider hypersurfaces of the form $S: y=F(a,b,x)$, with $\frac{\partial F}{\partial a}(0,0,0)\neq 0$, and define a para-CR structure on $S$ as the structure on $TS$ induced by the embedding, namely the two direction fields $TS\cap T\mathbb R^2_{xy}$ and 
$TS\cap T\mathbb R^2_{ab}$. The para-CR hypersurface is regular if the commutator of the two direction fields generates the missing direction in $TS$ at each point. This is analogous to Levi non-degeneracy of a CR manifold. 

In this paper we study normal forms of regular para-CR hypersurfaces, in the spirit of Chern and Moser \cite{Chern-Moser}. This leads to a normal form of second order ODE. We have learnt recently that I. Kossovskiy and D. Zaitsev \cite{KZ} obtained independently a result on the classification of second order ODE's. 
Furthermore, we also study the formal normal forms for  non regular hypersurfaces of finite type, following ideas in \cite{MR2288220}. Last but not least, we compute the infinitesimal automorphisms of  para-CR hypersurfaces of finite type, completing the study that was started in \cite{Ottazzi-Schmalz1}.

The paper is organized as follows. After establishing the notation, we study in Section~\ref{regular} the normal forms for regular hypersurfaces. First, we prove  in Theorem~\ref{th1} that the normal form can be achieved by 
 weighted jets of the function defining the surface, and then we show that the normalization is convergent in Theorem~\ref{convergence}. In Section~\ref{Section:ODE}, we interpret a surface  $S: y=F(a,b,x)$ as the space of solutions of a second order differential equation in one variable $y''(x)=B(x,y,y')$, with $a,b$ being parameters indicating the initial conditions for $y$ and $y'$. In Proposition~\ref{ODEnormal}, we show how the normal form for $S$ reflects into a normal form of the function $B$.
In Section~\ref{singular}, we consider the case of hypersurfaces $S$ that are not regular. We define a normal form for the jets of the defining equation and prove in Theorem~\ref{thm:singular} that every such jet can be put into normal form.
Finally, in Section~\ref{autos}, we apply the normal form to the study of the automorphisms of a hypersurface of finite type. 
The automorphisms for the model cases, in which $F$ is a polynomial, were studied in \cite{Ottazzi-Schmalz1}. Here we study the non model case.
In this setting, we show in Theorem~\ref{thm:autos} that there are no nontrivial isotropic automorphisms unless the Taylor expansion of the defining function $F$ in normal form consists of monomials of the type $(b^mx^n)^r$, where $m,n$ are fixed integers. In the latter case, we show that there is exactly  a $1$-parameter group of automorphisms.

\section{Notation}

We consider smooth hypersurfaces in $\mathbb R^4$ of the form $S: y= F(a,b,x)$, for which we assume $\frac{\partial F}{\partial a}(0,0,0)\neq 0$. This embedding distinguishes two direction fields $TS\cap T{\mathbb R}^2_{xy}$ and $TS\cap T{\mathbb R}^2_{ab}$ on $S$, which in local coordinates $x,a,b$ take the form 
$$
X=\frac{\partial}{\partial x},\qquad Y=\frac{\partial F}{\partial a} \frac{\partial}{\partial b}-  \frac{\partial F}{\partial b}  \frac{\partial}{\partial a}.
$$
We call a hypersurfaces in $\mathbb R^4$ with two distinguished vector fields $X,Y$ a {\it para-CR hypersurface}.
Sometimes it will be convenient to denote the partial derivatives of a function with the subscripts, in which case we shall write $F_{b},F_x,F_{bx},$ and so on.

We say that $S$ is of {\it finite type $k\geq 2$} if $k$ is the smallest integer such that $\frac{\partial^k F}{\partial^m b\partial^n x}(0,0,0)\neq 0$ for some $m,n>0$ such that $m+n=k$, provided such a $k$ exists. 
We say that $S$ is {\it regular} if $k=2$, otherwise we call $S$ {\it singular}. Given $S$ of finite type $k$, let $F_\ell(a,b,x)$ be the polynomial such that 
$$F(t^k a,tb,tx)= F_\ell(t^k a,tb,tx) + o(t^\ell),$$
as $t\to 0$, and call $F_\ell$ the {\it weighted $\ell$-jet} of $F$. Denote by $R_\ell=F-F_\ell$ the remainder of weight greater than $\ell$.
We say that a polynomial $P(a,b,x)$ has {\it weight} $\ell$ if $P(t^k a,tb,tx)=t^\ell P(a,b,x)$.
After a polynomial coordinate change,  we may write
$$
y=a+P(b,x)+R_k,
$$
where $P(b,x)= \sum_{i=1}^{k-1} \gamma_i b^ix^{k-i}$ with $\gamma_i \neq 0$ for some $1\leq i\leq k-1$, and $R_k= o(|a|+(|x|+|b|)^k)$. In particular, $F_k=a+P(b,x)$. Indeed, if 
$$
y= ga+f_1b+ \dots +f_k b^k +
h_1 x+\dots+ h_k x^k+ \sum_{i=1}^{k-1} \gamma_i b^ix^{k-i} +R_k,
$$
for some constants $g, h_1,\dots, h_k, f_1,\dots,f_k$, 
an appropriate coordinate change is $x^*=x$, $y^*=y-\sum_{i=1}^k h_k x^k$, $a^*=g a +\sum_{i=1}^k f_kb^k$ and $b^*=\gamma_m b$ leads to 
\begin{equation}\label{jeq}
y=a+b^m x^n + \sum_{i=m+1}^{k-1} \gamma_i b^ix^{k-i} +R_k,
\end{equation}
where we denoted by $m$ the smallest index for which $\gamma_m\neq 0$.

%
%
%

\section{Normal forms in the regular case}\label{regular}
In this section we consider smooth regular para-CR hypersurfaces in $\mathbb R^4$.
First, we discuss normal forms for regular surfaces. Second, we relate the normal form of regular hypersurfaces to normal forms for second order ODEs.

\subsection{Jet analysis and formal normal form}
Let $F_\ell=\sum_{\nu=2}^\ell P_\nu$ be the decomposition of $F_\ell$ into weighted homogeneous polynomials of weight $2\leq \nu\leq \ell$.

\begin{df}
We say that the regular hypersurface $S$ is in {\it jet normal form of order $\ell\ge 2$} if $F_2=P_2=a+bx$ and $P_\nu$ does not contain monomials of the type 
\begin{align}\label{norm1}
a^i x^j,&\quad a^ib^j, \quad 2i+j=\nu\\\label{norm2}
a^ibx^j,& \quad a^ib^jx,\quad  2i+j+1= \nu, \\\label{norm3}
a^ib^2x^2,&\quad a^ib^2x^3,\quad a^ib^3x^2,\quad a^ib^3x^3 
\end{align}
for $2<\nu\le \ell$.
\end{df}

In preparation of Theorem \ref{th1} below we introduce the Chern-Moser operator
$$T\colon (\eta,\alpha,\beta,\xi) \mapsto \eta- \alpha - x \beta - b\xi|_{y=a+bx}$$
where $\xi, \eta$ are functions (or formal power series) of $x,y$ and $\alpha,\beta$ are functions (or formal power series) of $a,b$. 
The argument $(\eta,\alpha,\beta,\xi)$ can be interpreted as the coefficients of a vector field
$$V=\eta \partial_y + \alpha \partial_a + \beta \partial_b +\xi \partial_x$$
and 
$$T(V) = V (y-a-bx)|_{y=a+bx}.$$
We are interested in studying the kernel and the image of $T$.
We assign to $\partial_y, \partial_a, \partial_b$ and  $\partial_x$ the weights $-2,-2,-1$ and $-1$ respectively.
Then $T$ maps a homogeneous vector field of weight $\ell-2$ to a homogeneous polynomial of weight $\ell$. It follows that the homogeneous components of both kernel and image belong to the kernel and the image respectively. Therefore it is sufficient to analyse kernel and image of the restrictions $T= T|_{V_\ell}$ , with $V_\ell=  \eta_{\ell} \partial_y + \alpha_\ell \partial_a + \beta_{\ell-1} \partial_b +\xi_{\ell-1} \partial_x$.
For $\ell\ge 5$ the kernel is trivial, which follows from the fact that the dimensions of preimage and
image coincide.
By straightforward calculation we obtain the results that are summarised in Tables \ref{tab1} and \ref{tab2}, 
where the dots stand for terms of bidegree $(i,j)$ in $(b,x)$ with $\min(a,b)\ge 2$ and $\max(a,b)\ge 4$.

\begin{table}
{\small
\begin{tabular}{|c|c|c|c|}
\hline
$\ell$ & $V$ & Kernel & Image \\ \hline
$0$& $\eta_{0} \partial_y + \alpha_0 \partial_a$ & $\partial_y +\partial_a$ & $\eta_0 - \alpha_0$\\ \hline
$1$& $\eta_{0}x \partial_y +\xi_{0} \partial_x+ \alpha_0b \partial_a+\beta_0 \partial_b $ & 
$x\partial_y +\partial_b, b\partial_a -\partial_b$& $(\eta_0-\beta_0)x- (\alpha_0+\xi_0)b$ \\ \hline
$2$& $(\eta_{0}x^2+\eta_1 y)  \partial_y +\xi_0 x \partial_x+$ &$y\partial_y+a\partial_a+b\partial_b,$ &$\eta_0x^2+(\eta_1-\beta_0-\xi_0)bx+$\\
& $+ (\alpha_0b^2+\alpha_1 a) \partial_a+\beta_0 b \partial_b $ &$b\partial_b-x\partial_x$ &$+(\eta_1-\alpha_1)a -\alpha_0b^2$ \\ \hline
$3$&$(\eta_{0}x^3+\eta_1 xy)  \partial_y +(\xi_0 x^2+\xi_1y) \partial_x+$  &$ab\partial_a+b^2\partial_b-y\partial_x,$&$\eta_0 x^3+(\eta_1-\xi_0)bx^2-(\xi_1+\beta_0)b^2x+  $\\ 
& $+ (\alpha_0b^3+\alpha_1 ab) \partial_a+(\beta_0 b^2+ \beta_1 a) \partial_b $ &$xy\partial_y+x^2\partial_x+a\partial_b$ &$+(\eta_1-\beta_1)ax-\alpha_0b^3-(\alpha_1+\xi_1)ab$\\\hline
$4$& $\sum_{j=0}^2 \eta_j x^{4-2j}y^j \partial_y+(\xi_0 x^3+\xi_1xy) \partial_x+ $&$y^2\partial_y+xy\partial_x+$ &$\eta_0x^4+(\eta_1-\xi_0)bx^3+\eta_1 ax^2+$ \\
& $+ \sum_{j=0}^2 \alpha_j b^{4-2j}a^j \partial_a+(\beta_0 b^3+ \beta_1 ab) \partial_b $ & $+a^2\partial_a+ab\partial_b$ & $+\eta_2b^2x^2+ (2\alpha_2-\xi_1-\beta_1)abx+$ \\ 
& & & $-\beta_0 b^3x+(\eta_2-\alpha_2) a^2-\alpha_1 ab^2-\alpha_0b^4$\\\hline
\end{tabular}
}
\caption{Kernel and image of $T_\ell$ for $0\le \ell\le 4$. }\label{tab1}
\end{table}
\bigskip

\bigskip

\begin{table}
{\small
\begin{tabular}{|c|c|c|}
\hline
$\ell$&$V$& Image\\\hline
$\ell=2k+1$ &$\sum_{j=0}^{k} \eta_j x^{\ell-2j}y^j \partial_y+  $& $\sum_{j=0}^{k-2}\xi_ja^jbx^{\ell-2j-1}+\sum_{j=0}^k \alpha_j a^jb^{\ell-2j}+$ \\
$k\ge2$ & $+ \sum_{j=0}^{k}\xi_j x^{\ell-2j-1}y^j \partial_x+ $& $+\sum_{j=0}^{k-1} \beta_j a^jb^{\ell-2j-1}x+\sum_{j=0}^{k-1}\eta_ja^jx^{\ell-2j}+$ \\
&  $\sum_{j=0}^{k} \alpha_j a^jb^{\ell-2j} \partial_a+  $&  $-(\xi_k+\alpha_k) a^kb -(k\xi_k+\beta_{k-1})a^{k-1}b^2x-\binom{k}{2}\xi_ka^{k-2}b^3x^2$\\
& $+ \sum_{j=0}^{k}\beta_j a^jb^{\ell-2j-1}\partial_b $& $+(\eta_{k}-\beta_{k}) a^{k} x+ (k \eta_{k}-\xi_{k-1})a^{k-1}bx^2+$\\
& & $+[\binom{k}{2}\eta_k -(k-1)\xi_{k-1}]a^{k-2}b^2x^3+\dots$ \\
\hline
$\ell=2k$ &$\sum_{j=0}^{k} \eta_j x^{\ell-2j}y^j \partial_y+  $& $\sum_{j=0}^{k-2}\xi_ja^jbx^{\ell-2j-1}+\sum_{j=0}^{k-1} \alpha_j a^jb^{\ell-2j}+$\\
$k\ge3$ & $+ \sum_{j=0}^{k-1}\xi_j x^{\ell-2j-1}y^j \partial_x+ $&  $+\sum_{j=0}^{k-2} \beta_j a^jb^{\ell-2j-1}x+\sum_{j=0}^{k-1}\eta_ja^jx^{\ell-2j}+$\\
&  $\sum_{j=0}^{k} \alpha_j a^jb^{\ell-2j} \partial_a+  $& $(k\eta_k-\xi_{k-1}-\beta_{k-1}) a^{k-1}bx +\left(\binom{k}{2}\eta_{k}-k\xi_{k-1}\right)a^{k-2}b^2x^2+$\\
& $+ \sum_{j=0}^{k-1}\beta_j a^jb^{\ell-2j-1}\partial_b $&$+\left(\binom{k}{3}\eta_k-\binom{k-1}{2}\xi_{k-1}\right)a^{k-3}b^3x^3+(\eta_k-\alpha_k)a^k +\dots$ \\
\hline
\end{tabular}
}
\caption{Image of $T_\ell$ for $\ell\ge5$.}\label{tab2}
\end{table}
\bigskip

We can now prove our first result.

\begin{theorem}\label{th1}
Let $S: y= F(a,b,x)$ be a regular hypersurface in $\mathbb R^4$. For every $\ell$ there is a polynomial mapping that takes $F_\ell$ into normal form. Moreover, if $F_\ell$ is in jet normal form, then $F_{\ell+1}$ may be taken into jet normal form by a mapping $\Phi_{\ell+1}$ defined component-wise by
\begin{align}\label{mapn}
X&=x+ \xi_\ell(x,y) & \qquad
Y&=y+ \eta_{\ell+1}(x,y)\\\nonumber
A&=a+ \alpha_{\ell+1}(a,b)& \qquad
B&=b+ \beta_\ell(a,b)
\end{align}
where $\xi_\ell$ and $\beta_\ell$ are polynomials of weight $\ell$ and $\eta_{\ell+1}$ and $\alpha_{\ell+1}$ are polynomials of weight $\ell+1$. The mapping $\Phi_{\ell+1}$ is unique if $\ell>4$.
\end{theorem}

\begin{proof}
We use induction by $\ell\ge 2$. By \eqref{jeq} we may write $F_2=a+bx$.
Suppose next that $F_\ell$ is in normal form. Assume that after applying the mapping \eqref{mapn} on
$$y=a+bx + \sum_{\nu=3}^{\ell+1} P_\nu + R_{\ell+1}$$
we obtain
$$Y=A+BX + \sum_{\nu=3}^{\ell+1} P^*_\nu + R^*_{\ell+1}.$$
This yields the identity
\begin{equation*}
y+\eta_{\ell+1}- a-\alpha_{\ell+1} - (b+\beta_\ell)(x+\xi_\ell) -\sum_{\nu=3}^{\ell+1} P^*_\nu(a+\alpha_{\ell+1},b+\beta_\ell,x+\xi_\ell)+R^*_{\ell+1}|_{y=a+ bx +\sum_{\nu}^{\ell+1} P_{\nu}+R_{\ell+1}}= 0.
\end{equation*}
Isolating the terms of weight $\ell+1$ gives
$$P_{\ell+1}+\eta_{\ell+1}-\alpha_{\ell+1} - x \beta_\ell- b \xi_\ell - P^*_{\ell+1}|_{y=a+ bx}= 0.$$

This implies that we can eliminate exactly those terms in $P_{\ell+1}$ that belong to the image of the Chern-Moser operator $T_{\ell+1}$. Since the space spanned by the terms \eqref{norm1} to \eqref{norm3} is complementary to the image of $T_{\ell+1}$, from Tables \ref{tab1} and \ref{tab2} these terms can be eliminated by the mapping \eqref{mapn}. Notice that $T_{\ell+1}$ does not affect $F_{\ell}$. 
 \end{proof}
The previous theorem can be restated in terms of formal power series in the following way.

\begin{df}
Given a regular para-CR hypersurface $S: y=F(a,b,x)$, we say that $S$ is in if 
it is in jet normal form of order $k$ for any $k$.
\end{df}

\begin{cor}
For any formal power series equation  $y=a+bx +\sum_{j=3}^\infty f_j{(a,b,x)}$ there is a formal power series coordinate change
\begin{align*}
X&=x+ \sum_{\nu=2}^\infty \xi_\nu(x,y) & A&=a+ \sum_{\nu=2}^\infty\alpha_{\nu+1}(a,b)\\
Y&=y+ \sum_{\nu=2}^\infty\eta_{\nu+1}(x,y) & B&=b+ \sum_{\nu=2}^\infty\beta_{\nu}(a,b)
\end{align*}
that takes the summands of the formal defining equation into formal normal form.
\end{cor}
   


Following Chern and Moser~\cite{Chern-Moser} we give next a geometric construction of the normalisation. 

\begin{df}
We say that a para-CR hypersurface $S\colon y=F(x,a,b)$ is in normal form if $y=a+bx+ f(x,a,b)$ where
\begin{itemize}
\item[(i)] $f|_{x=0}=f|_{b=0}=0$
\item[(ii)] $\partial_x f|_{x=0}= \partial_b f|_{b=0}=0$
\item[(iii)] $\partial_{xxbb}f|_{x=b=0}=0$
\item[(iv)] $\partial_{xxxbb}f|_{x=b=0}= \partial_{xxbbb}f|_{x=b=0}=0$
\item[(v)] $\partial_{xxxbbb}f|_{x=b=0}=0$. 
\end{itemize}
 \end{df}
 
{\bf Remark.} If the defining equation of a para-CR hypersurface has a Taylor series and is in normal form then the Taylor series has formal normal form. If a para-CR hypersurface is real-analytic then the normal form conditions are equivalent to the formal normal form conditions.  

\begin{theorem}\label{convergence}
Let $S: y= F(x,a,b)$ be a regular para-CR hypersurface in $\mathbb R^4$. Then $S$ may be taken into normal form by a smooth change of coordinates.
\end{theorem}
\begin{proof}
From \eqref{jeq}, we may write
$$S\colon y= a+bx + f(a,b,x),$$
with $f(0,0,0)=df(0,0,0)=d^2f(0,0,0)=0$, where by $df$ and $d^2f$ we denote the collection of all first order and second order partial derivatives of $f$. 
We  normalise this equation by a finite sequence of mappings.
For each consecutive mapping we shall denote the old coordinates by $x,y,a,b$ and  the new coordinates by $X,Y,A,B$, and  old and new defining functions by $f$ and $f^*$ respectively.
\begin{enumerate}
\item[Step 1.]
The first mapping is given by its inverse:
\begin{align*}
x&= X+ p(Y) & a&= \psi(A)\\
y&= q(Y)& b&= B + \pi(A)
\end{align*}
with $p(0)=\pi(0)=\psi(0)=0$, $q'(0)=\psi'(0)=1$. We assume that the curve 
$\gamma\colon x=p(t), y=q(t), a=\psi(t), b=\pi(t)$ belongs to $S$, that is
$$q=\psi+ \pi p+ f(\psi,\pi,p).$$
 Then $\gamma$ is mapped to 
$\gamma^*\colon Y=A=t, X=B=0$, which belongs to the new hypersurface 
$S^*\colon Y=A+BX + f^*(A,B,X)$. It follows that $f^*(A,0,0)\equiv 0$. Later we shall impose additional conditions on the curve $\gamma$, which will make it an analogue of a Chern-Moser chain.

\item[Step 2.]
In the second step we seek a change of coordinates of the form 
\begin{align*}
X&= x & A&= a+ h(a,b)\\
Y&= y+g(x,y)& B&= b
\end{align*}
that maps $S$ to $Y=A+BX+f^*{(A,B,X)}$ in such a way that $f^*(A,B,0)= f^*(A,0,X)= 0$.
In order to achieve this, first observe that the equation $Y=A+BX+f^*{(A,B,X)}$ is equivalent to 
\begin{equation}\label{neweq}
 y+g(x,y)= a+ h(a,b)+bx+ f^*(a+h(a,b),b,x)|_{y=a+bx+f(a,0,x)+f(a,b,0)+\tilde{f}}.
\end{equation}

First, imposing $f^*(A,B,0)= f(a+h(a,b),b,0) = 0$ and  $x=0$ in \eqref{neweq}, 
$$
f(a,b,0)+g(0,y)=h(a,b),
$$
whence $g(0,y)=0$ and $h(a,b)=f(a,b,0)$. In particular, since $f(a,0,0)=0$ from Step 1, $h(a,0)=0$.

Second, imposing $f^*(A,0,X)= f(a+h(a,b),0,x) = 0$ and  $b=0$ in \eqref{neweq}, and since $h(a,0)=0$,
$$
f(a,0,x)+g(x,y)=0,
$$
whence $g(x,y)=-f(a,0,x)$, where $a=a(x,y)$ is the solution of the implicit function with parameter $x$
$$y=a+f(a,0,x)$$ with $a(0,0)=0$.

%

%

\item[Step 3.]
In the third step we  use the mapping
\begin{align*}
X&= x & A&= a\\
Y&= y& B&= C(a)b\\
\end{align*}
with $C(a)=(1+f_{bx}(a,0,0))^{-1}$, in order to achieve $f^*_{BX}(A,0,0)=0$. 
The new defining function $f^*$ has the properties $f^*(A,B,0)=f^*(A,0,X)=f^*_{BX}(A,0,0)=0$. 

It follows
that
$$f^*(A,B,X)=f^*(A,0,X)+ f^*_B(A,0,X)B+ \frac{1}{2} f^*_{BB}(A,B,X)B^2=f^*_B(A,0,X)B + \phi(A,B,X)B^2.$$
Furthermore,
\begin{align*}
f^*_B(A,0,X)&=f^*_B(A,0,0)+ f^*_{BX}(A,0,0)X + O(X^2)\\
\phi(A,B,X)&= \phi(A,B,0) + \phi_X(A,B,0)X + O(X^2)
\end{align*}
with $f^*_B(A,0,0)=0$ and $\phi(A,B,0)=0$.
Hence,
\begin{equation}\label{expans}
f^*(A,B,X)= f^*_{BX}(A,0,X)BX  + \phi_X(A,B,0)B^2 X + O(B^2 X^2).
\end{equation}

\item[Step 4.]
The fourth step consists in applying the transformation
\begin{align*}
X&= x& A&= a\\
Y&= y & B&= b+f_x(a,b,0)
\end{align*}
which leads to $f^*_X(A,B,0)=0$.
%
%
\item[Step 5.]

Next, in order to achieve $f^*_B(A,0,X)=0$, we apply
\begin{align*}
X&= x+ f_b(a,0,x)& A&= a\\
Y&= y & B&= b,
\end{align*}
where $a=a(x,y)$ is the solution of the implicit function with parameter $x$
$$y=a+f(a,0,x)$$ with $a(0,0)=0$.

%

\item[Step 6.]
The mapping of the next step is again given by its inverse:
\begin{align*}
x&= C(Y)X& a&= A\\
y&= Y& b&= B/ C(A)
\end{align*}
which acts on the term $f_{22}b^2x^2= \frac{1}{4}f_{bbxx}(a,0,0)b^2x^2$ of the defining function $f(a,b,x)= \frac{1}{4}f_{bbxx}(a,0,0)b^2x^2+ \tilde{f}_1+\tilde{f}_2$, where $\tilde{f}_1$ is divisible by $b^3x^2$ and $\tilde{f}_2$ is divisible by $b^2x^3$. By chosing $C(A)$  as the solution of the ODE
$$\frac{C'(A)}{C(A)}=- f_{22}$$
with initial condition $C(0)=1$ we can eliminate the term $f_{22}$. By solving the ODE with another RHS we can modify the term $f_{22}$  to any other prescribed function, which will be used in Step 8.

\item[Step 7.]
The last transformation is given by the inverse of
\begin{align*}
x&= \sqrt{h'(Y)}X & a&= h(A)\\
y&= h(Y) & b&= \sqrt{h'(A)}B\\
\end{align*}
that we perform to eliminate the term $\frac{1}{36}f_{bbbxxx}(a,0,0,)b^3x^3$. This can be achieved by choosing $h$ as a solution of
$$h'''- \frac{3}{2} \frac{(h'')^2}{h'}+\frac{1}{3}f_{bbbxxx}(a,0,0)=0$$
with initial conditions $h(0)=0$, $h'(0)=1$.

\item[Step 8.]
It remains to choose the chain $\gamma$ in the first step so that the terms
$\frac{1}{12}f_{bbbxx}(a,0,0,)b^3x^2$ and $\frac{1}{12}f_{bbxxx}(a,0,0,)b^2x^3$ disappear. This can be achieved by another transformation that eliminates those terms.
Assume that $S$ is given by
$$y=a+bx +f_{22} b^2x^2 +f_{23} b^2x^3 + f_{32}b^3x^2 +\dots$$
where $f_{23}= f_{bbxxx}|_{b=x=0}$ and $f_{32}= f_{bbbxx}|_{b=x=0}$,
and apply a mapping of the form
\begin{align*}
X&=p(y)+x+ T_2(y)x^2+ T_3(y)x^3+\dots\\
Y&=q(y) +g_1(y)x+  g_2(y)x^2+ g_3(y)x^3+\dots\\
A&=\psi(a)+h_1(a)b+  h_2(a)b^2+ h_3(a)b^3+\dots\\ 
B&= \pi(a)+b+ S_2(a)b^2+ S_3(a)b^3+\dots
\end{align*}
to eliminate $f_{23}$ and $f_{32}$, preserving the conditions on $f_{k0},f_{0k},f_{k1},f_{1k}$. For computational convenience we impose a parametrisation on the chain $\gamma$ such that 
$$q'=1+p'\pi.$$
This condition is not essential as it can be undone by consecutive application of Step 7.

Plugging the transformation into the desired resulting equation and restricting to the hypersurface yields
\begin{multline*}
q+q'(bx+f_{22}b^2x^2+f_{23}b^2x^3+f_{32}b^3x^2)+\frac{1}{2}q''(b^2x^2+2f_{22}b^3x^3)+\dots\\
g_1x+ g'_1(bx^2+ f_{22}b^2x^3)+ \frac{1}{2}g''_1b^2x^3+\dots + g_2x^2+ g'_2bx^3+\dots\\
=\psi+ h_1b+ h_2b^2+\dots+ [\pi+b+S_2b^2+S_3b^3+\dots]\times\\
\times [p+p'(bx+f_{22}b^2x^2+f_{23}b^2x^3+f_{32}b^3x^2) +\frac{1}{2}p''b^2x^2+x+T_2x^2+T_2'bx^3+T_3x^3+\dots]
\end{multline*}


Hence,
%
\begin{align*}
g_1&=\pi & h_1&=-p\\
g_k&=\pi T_k & h_k&=-p S_k\\
g_k'&=T_{k+1}+\pi T_k'& 0&=S_{k+1}+p'S_k
\end{align*}
and therefore
\begin{align*}
T_k&=(\pi')^{k-1} & S_k&=(-p')^{k-1}\\
g_k&=\pi (\pi')^{k-1} & h_k&= -p (-p')^{k-1}.
\end{align*}

This shows that the mapping above can be written in the closed form
\begin{align*}
X&=p(y)+ \frac{x}{1-x\pi'(y)} &\qquad Y&=q(y)+ \frac{x\pi(y)}{1-x\pi'(y)}\\
A&=\psi(a)- \frac{b\psi(a)}{1+bp'(a)} & \qquad
B&= \pi(a)+\frac{b}{1+bp'(a)}.
\end{align*}

Finally, we find
\begin{align*}
f_{22}&=S_2T_2 + \frac{1}{2}\pi p'' - \frac12 q''=-p'\pi' + \frac{1}{2}\pi p'' - \frac12 q''=-\frac32 p'\pi'\\
f_{32}&= p'f_{22}  +\frac{1}{2}p'' + S_3 T_2=\frac{1}{2} (p''- (p')^2\pi')\\
f_{23}&=-g_1'f_{22}-\frac12 g_1'' +T_2'+ S_2 T_3=\frac{1}{2} (\pi''+ (\pi')^2p').
\end{align*}

The chain equations are therefore
$$p''- (p')^2\pi'=2f_{32}, \qquad \pi''+ (\pi')^2p'=2f_{23}.$$
The new equation has no term $f^*_{22}$ if $f_{22}$ in the old equation had been modified to $f_{22}=-\frac32 p'\pi'$ by prior application of Step 6.
\end{enumerate}

\end{proof}

\subsection{Second order  ODE}\label{Section:ODE}
We notice that the solutions of a second order ODE in one variable can be interpreted as a para-CR hypersurface in ${\mathbb R}^4$, when we treat the initial conditions as parameters. In this section we study the consequences of the normal form for regular hypersurfaces in ${\mathbb R}^4$ in the second order ODEs of which they represent the solutions.

Let $$y''+r(x)y'+s(x) y=0$$ be a homogeneous linear second order ODE and let $f_1(x)$ and $f_2(x)$ be fundamental solutions with $f_1(0)=f_2'(0)=1$ and $f_1'(0)=f_2(0)=0$. Then the general solution is 
$$y= af_1+bf_2= a+bx +\dots$$
It is easy to check that the Lie invariants vanish and therefore the ODE is equivalent to $\tilde{y}''=0$.
This fact is consistent with our previous analysis.
In fact, the coordinate transformation
$$\tilde{y}=\frac{y}{f_1(x)}, \quad \tilde{x}= \frac{f_2(x)}{f_1(x)}$$
takes the solution manifold to the normal form
$$\tilde{y}=a+ b\tilde{x}.$$

The original ODE can be written as
$$\tilde{y}''=\frac{ \begin{vmatrix} y'' & {f_1}''\\y& {f_1} \end{vmatrix} \begin{vmatrix} {f_2}' & {f_1}'\\{f_2}& {f_1} \end{vmatrix} -
\begin{vmatrix} {f_2}'' & {f_1}''\\{f_2}& {f_1} \end{vmatrix} \begin{vmatrix} y' & {f_1}'\\y& {f_1} \end{vmatrix}}{\begin{vmatrix} {f_2}' & {f_1}'\\{f_2}& {f_1} \end{vmatrix}}=\left(\frac{W(y,{f_1})}{W({f_2},{f_1})}\right)'W({f_2},{f_1})=0$$

where $W({f_2},{f_1})=\begin{vmatrix} {f_2}' & {f_1}'\\{f_2}& {f_1} \end{vmatrix}$ is the Wronskian of the two solutions ${f_1}, {f_2}$ etc.
  
In particular, let $y''+ry'+sy=0$ be a homogeneous linear ODE with constant coefficients and with general solution
$$y=Ae^{\lambda x} + B e^{\mu x}.$$
Assume the $\lambda, \mu$ are distinct solutions of the characteristic equation. Then
$$y=a\frac{\mu e^{\lambda x}- \lambda e^{\mu x}}{\mu-\lambda} +b \frac{e^{\lambda x}- e^{\mu x}}{\mu-\lambda}=a+bx +a\left(\frac{\mu e^{\lambda x}- \lambda e^{\mu x}}{\mu-\lambda} -1\right)+ b
\left(\frac{e^{\lambda x}- e^{\mu x}}{\mu-\lambda}-x\right).$$

The normalising coordinate change according to the procedure from above is
$$\tilde{x}=\frac{e^{\mu x}- e^{\lambda x}}{\mu e^{\lambda x}- \lambda e^{\mu x}},\qquad 
\tilde{y}= \frac{(\mu-\lambda)y}{\mu e^{\lambda x}- \lambda e^{\mu x}}$$
together with $\tilde{a}=a$ and $\tilde{b}=a$.
This point transformation also transforms the original ODE to $\tilde{y}''=0$.
Indeed, 
\begin{align*}
\tilde{y}'&=\frac{d\tilde{y}}{d\tilde{x}}=\frac{\lambda e^{\mu x}(y'-\mu y)-\mu e^{\lambda x}(y'-\lambda y)}{(\lambda-\mu)e^{(\lambda+\mu)x}}\\
\tilde{y}''&=\frac{d\tilde{y}'}{d\tilde{x}}=\frac{(\lambda e^{\mu x}-\mu e^{\lambda x})^3}{(\lambda-\mu)^3e^{2(\lambda+\mu)x}}(y''-(\lambda+\mu)y'+\lambda\mu y).
\end{align*}

The inverse transformation is
$$x=\frac{1}{\mu-\lambda} \log \frac{1+\mu \tilde{x}}{1+\lambda \tilde{x}}, \qquad y=\left[\frac{\mu}{\mu-\lambda}\left(\frac{1+\mu\tilde{x}}{1+\lambda\tilde{x}}\right)^{\frac{\lambda}{\mu-\lambda}}-\frac{\lambda}{\mu-\lambda}\left(\frac{1+\mu\tilde{x}}{1+\lambda\tilde{x}}\right)^{\frac{\mu}{\mu-\lambda}}\right]\tilde{y}.$$



We consider now  the general second order differential equation
\begin{equation}\label{ODE}
\begin{cases}
y''(x)=B(x,y,y')\\
y(0)=a\\
y'(0)=b
\end{cases}
\end{equation}
where $B$ is a function that we suppose to be analytic in $(x,y,y')$. The solution is some analytic function $y=F(x,a,b)$.
From what we showed above, we can apply smooth transformations to $F$ so that it appears in normal form. The normal form of $F=a+bx+f(x,a,b)$ has an immediate correspondence for the solution $y$ of \eqref{ODE}, and in turn it entails a corresponding normal form on $B(x,y,y')$. The investigation of the resulting normal form for the function $B$ is the content of the following statement.

\begin{proposition}\label{ODEnormal}
Consider the differential equation \eqref{ODE}, and assume that $B$ is analytic in $(x,y,y')$.
Then, after a smooth change of coordinates, $B(x,y,y')=\sum B_{ij}(y)x^i (y')^j$ satisfies
 $B_{ij}=0$ for pairs of indices
$(i,0)$, $(i,1)$ with $i\ge 0$, $(0,2)$, $(0,3)$, $(1,2)$, and $(1,3)$. In particular,
$$B= (y')^4 \sum_{j=0}^\infty (B_{0,j+4}(y) + x B_{1,j+4}(y)) (y')^j + \sum_{i,j \ge2} B_{ij}(y)x^i (y')^j.$$  
\end{proposition}

\begin{proof}
We write $y^{i}_{j}=\frac{\partial^{i+j}y}{{\partial_x}^i{\partial_b}^j}$, $y^i_0=y^i$, and $y^0_j=y_j$. Then the solution $F(x,a,b)$ of \eqref{ODE}  is in normal form if and only if
\begin{align}
&y^{i}(x,a,0)=0, \forall i\geq 0 \label{one}\\
&y_{j}(0,a,b)= 0, \forall j\geq 0, \label{two}\\
&y_1(x,a,0)=x \label{three}\\
&y^1_1(0,a,0)=1 \label{four}\\
&y^i_1(x,a,0)=0,\forall i\geq 2 \label{five}\\
&y^1_j(0,a,b)=0,\forall j\geq 2 \label{six}\\
&y^2_2(0,a,0)=y^3_2(0,a,0)=y^2_3(0,a,0)=y^3_3(0,a,0)=0. \label{eight}
\end{align}
We have $B(x,a,b)=F_{xx}(x,a(x,y,y'),b(x,y,y'))$ where $a(x,y,y')$ and $b(x,y,y')$ are the solutions of the implicit system of equations
\begin{align*}
y&= a+bx+ f(x,a,b),\\
y'&= b+ f_x(x,a,b).
\end{align*}

Assume that 
\begin{align*}
a(x,y,y')= \sum_{j=0}^\infty a_j(y,y') x^j,\\
b(x,y,y')= \sum_{j=0}^\infty b_j(y,y') x^j.
\end{align*}
Then
\begin{align*} 
y\equiv& \sum_{j=0}^\infty a_j(y,y') x^j + \sum_{j=0}^\infty b_j(y,y') x^{j+1} + f(x, \sum_{j=0}^\infty a_j(y,y') x^j,\sum_{j=0}^\infty b_j(y,y') x^j),\\
y'\equiv&  \sum_{j=0}^\infty b_j(y,y') x^{j} + f_x(x, \sum_{j=0}^\infty a_j(y,y') x^j,\sum_{j=0}^\infty b_j(y,y') x^j).
\end{align*}

Since $f$ is divisible by $x^2$ and $f_x$ is divisible by $x$, we readily find
$$a_0=y, \qquad b_0=y', \qquad a_1=-y', \qquad b_1= -f_{xx}(0,y,y')=o((y')^3).$$

Consider the $x^2$-term:
\begin{align*}
0=& a_2x^2 + b_1 x^2 + \frac{1}{2}f_{xx}(0,y,y'),\\
0= &b_2x^2 + (\frac{1}{2} f_{xxx}(0,y,y') + f_{xxa}(0,y,y')a_1+ f_{xxb}(0,y,y')b_1)x^2.
\end{align*}
So
$$a_2=\frac{1}{2}f_{xx}(0,y,y')=o((y')^3)$$
and
$$a_n=-b_{n-1}+ \frac{1}{n}b_{n-1}=\frac{1-n}{n}b_{n-1}.$$

This means
$$a=y+{\mathfrak b}-x(y'+{\mathfrak b}_x),$$
where $\mathfrak b(x,y,y')=\int_0^x b(t,y,y')dt$.

The function $\mathfrak b$ satisfies the implicit equation
$$\mathfrak b= xy' -f(x,y+\mathfrak b-x(y'+\mathfrak b_x),\mathfrak b')$$
with $\mathfrak b(0,y,y')\equiv 0$.

If $\phi=\mathfrak b- xy'$, then
$$\phi=-f(x, y-xy'+\phi-x\phi_x,y'+\phi_x)$$
with $\phi(0,y,y')=\phi_x(0,y,y')=0$.

Let $\phi=\sum_{n=2}^\infty \phi_n x^n$. Then
$$\phi_n= -\sum_{k=0}^n\frac{\partial^n f}{\partial x^{n-k}\partial a^{k}}(0,y,y')\frac{(y')^k}{(n-k)!k!} +\dots$$
where the $\dots$ terms contain at least one factor of $\phi_\nu$ or $\phi'_\nu$ with $\nu<n$ and no terms $\phi_\nu$ or $\phi'_\nu$ with $\nu\ge n$. It can be proved by induction that all $\phi_n$ are divisible by $(y')^2$.
\end{proof}

{\bf Remark.} We show that $B\equiv 0$ if the two Lie-Tresse invariants  (see, e.g., \cite{Tresse, MR2024797}) vanish. The first invariant is $\frac{\partial^4 B}{(\partial y')^4}$. Its vanishing yields
$$B_{i,j+4}=0 \text{ for } i,j\ge 0.$$

Now vanishing of the order 0 component in $y'$ of the second invariant implies 
$$B_{i2}=0 \text{ for } i\ge 2$$
and vanishing of the order 1 component in $y'$ of the second invariant implies 
$$B_{i3}=0 \text{ for } i\ge 2.$$
Therefore $B\equiv 0$.


%
%

%

\section{Singular case}\label{singular}


It is known from the results by M. Kol\'a\v{r} \cite{MR2983069} in the CR-case that one cannot expect a convergent (smooth) normal form for singular hypersurfaces. Below we present a formal normal form construction similar to Kol\'a\v{r}'s \cite{MR2189248} formal normal form for finite type CR-manifolds. 

We start with a para-CR hypersurface $S\colon y=F(a,b,x)$ in $\mathbb R^4$ of finite type.
From \eqref{jeq}, we may write 
$$
F(a,b,x)=a +b^mx^n +\sum_{j=m+1}^{k-1} \gamma_jb^jx^{k-j} + \sum_{\nu=k+1}^\ell P_\nu(a,b,x) + R_\ell,
$$
for every $\ell>k$, where $m$ is the minimal number such that $\gamma_m\neq0$ and  $n=k-m$. 
 In particular, 
$$
F_\ell(a,b,x)= a +b^mx^n +\sum_{j=m+1}^{k-1} \gamma_jb^jx^{k-j} + \sum_{\nu=k+1}^\ell P_\nu(a,b,x).
$$
In contrast to the regular case studied in Section~\ref{regular}, we shall consider here the case $k>2$.
Similar to the approach in \cite{MR1668718} we use a single monomial $b^mx^n$ in the normalisation procedure.
\begin{df}
We say that $S$ is in {\it normal form of order $\ell\geq k$} with respect to $b^mx^n$ if $F_k = a +b^mx^n +\sum_{j=m+1}^{k-1} \gamma_jb^jx^{k-j}$ and for every $k<\nu \le \ell$, 
$P_\nu$ does not contain monomials of the type
\begin{align*}
&a^i x^j, \quad a^ib^j, \quad  ki+j =\nu, \\
& a^i b^m x^{n-1+j}, \quad a^ib^{m-1+j}x^n, \quad  k(i+1)+j-1=\nu, \\
& a^ib^{2m}x^{2n}, \quad a^j b^{3m} x^{3n}, \quad  k(i+2)=k(j+3)=\nu,\\
& a^i bx^{2n}, \text{ if } m=1,\\
& a^i b^{2m} x, \text{ if } n=1.
\end{align*}
\end{df}

\begin{theorem}\label{thm:singular}
Let $S: y=F(a,b,x)$ be a para-CR hypersurface of finite type $k>2$. For every $\ell>k$ there is a polynomial mapping that takes $F_\ell$ into normal form.
\end{theorem}


\begin{proof}
Throughout the proof, the letter $\gamma$ will denote a real number that depends only on indices $i,j$. 
We start with the case $m,n>1$.

First, using mappings of the form 
$$x\mapsto x,\quad y\mapsto y, \quad a \mapsto a+ \gamma a^i b^j, \quad b\mapsto b$$
with $ki+j>k$, we may simplify the monomials in $P_\nu$ of the form $a^i b^j$, only by affecting other higher order terms in the equation.


Second, mappings of the kind
$$x\mapsto x,\quad y\mapsto y, \quad a \mapsto a, \quad b\mapsto b+ \gamma a^i b^j,$$
where $ki+j>1$
affect the monomials $a^i b^{m-1+j}x^n$ and  terms of higher order.
Hence, the monomials of the form $a^ib^{m-1+j}x^n$ with  $k(i+1)+j-1>k$  can be eliminated.


Third, transformations
$$x\mapsto x,\quad y\mapsto y+ \gamma x^j y^i, \quad a \mapsto a, \quad b\mapsto b$$
with $ki+j>k$ affect the monomials $a^i x^j$ and higher order terms. So, terms  $a^i x^j$ with $ki+j>k$ can be simplified. However, for $j=0$ this has been already achieved in the first step. We might then use this mapping to eliminate terms of the form $a^{i-1}b^m x^{n}$, but these terms have been eliminated in the second step. 
Therefore we use this mapping to eliminate monomials of the form $a^{i-2}b^{2m}x^{2n}$.

Fourth,
$$x\mapsto x+ \gamma x^j y^i,\quad y\mapsto y, \quad a \mapsto a, \quad b\mapsto b$$
with $j+ki>1$ 
affect $a^i b^mx^{n-1+j}$, which has already been eliminated for $j=1$. The next affected term $a^{i-1}b^{2m}x^{2n}$ has also been eliminated. Therefore, we use this mapping to eliminate monomials of the form $a^{i-2}x^{3m}b^{2n}$.
\\

Suppose now that  $m=1$. The difference to the case $m>1$ is that the terms $a^ix^n$ can be eliminated in two ways, either by 
$$x\mapsto x,\quad y\mapsto y, \quad a \mapsto a, \quad b\mapsto b+ \gamma a^i$$
or by
$$x\mapsto x,\quad y\mapsto y+ \gamma x^n y^i, \quad a \mapsto a, \quad b\mapsto b.$$ 
Since the first mappings affect only those terms, we can use the second mapping to eliminate the monomials $a^i bx^{2n}$  instead. 

Similarly, if $n=1$, the terms of the form $a^i b^m$  can be eliminated by
$$x\mapsto x+ \gamma y^i,\quad y\mapsto y, \quad a \mapsto a, \quad b\mapsto b$$
or by
$$x\mapsto x,\quad y\mapsto y, \quad a \mapsto a+\gamma a^i b^m, \quad b\mapsto b.$$ 
In this case we use the first mapping to further eliminate monomials of the form $a^ib^{2m} x$.
%
\end{proof}

We remark that there is clearly a certain freedom of choosing the conditions. For example, one could look at other terms in $P$, if there are any, such as the terms $a^i b^{m-1}x^n$ and $a^ib^mx^{n-1}$.


\section{Application to automorphisms of para-CR hypersurfaces of finite type }\label{autos}

Let $S\colon y=a+ P(b,x) + f(a,b,x)$ be a hypersurface of finite type. 
In \cite{Ottazzi-Schmalz1}, the authors study the automorphisms of $S$ in the model case, that is when $f=0$. 
We complete here the study of the automorphisms in the general case using the properties of mappings that preserve our Chern-Moser type normal form and the fact that isotropic automorphisms of a hypersurface in normal form preserve normal form.     
If $\chi$ is an infinitesimal automorphism with $\chi(0)=0$ then  $\exp t \chi$ is a one parametric family of isotropic automorphisms. Suppose that $S$ is given in normal form. Then $\exp t \chi$ preserves the normal form and so does $\chi$.

\begin{theorem}\label{thm:autos}
Let $S\colon y=a+ P(b,x) + f(a,b,x)$ be a para-CR hypersurface of finite type, and suppose that $f$ is not constant.
The linear vector field $\chi = nb\partial_b- mx\partial_x$ is an infinitesimal automorphism of $S$ with $\chi(0)=0$ if and only if 
$P(b,x)=b^m x^n$ and the Taylor expansion of $f(a,b,x)$ with respect to $b$ and $x$ consists of monomials of the type $\# (b^mx^n)^r$.
Moreover, $\chi$ and its multiples are the only isotropic infinitesimal automorphisms that may occur.
\end{theorem}
\begin{proof}
We suppose first that  $P$ is not a monomial and show that there are no nontrivial infinitesimal automorphisms in this case.
The infinitesimal isotropic automorphisms of the model case $f=0$ are linear of the form \cite{Ottazzi-Schmalz1}[Theorem 1]
$$\chi_0= x\partial_x + b\partial_b + ka\partial_a + ky\partial_y.$$
It follows that an isotropic infinitesimal automorphism $\chi$ of a deformation $S$ of a model $S_0$ is a deformation $\chi_0+\tilde{\chi}$ of $\chi_0$ by higher order terms  $\tilde{\chi}$ (h.o.t. for short). This implies that the first  deformation terms of $S$ are preserved by $\chi_0$, which is impossible, because it has a higher weight than $k$. Therefore para-CR hypersurfaces such that $P$ is not a monomial and that are not equivalent to the model itself have no nontrivial  isotropic infinitesimal automorphisms.
\\

Consider next the case when $P=b^mx^n$ $(k=m+n>2)$ is a monomial. Assume that $S$ is given in the form \eqref{jeq} and $\chi$ is an isotropic infinitesimal automorphism with trivial linear part, i.e.
$$\chi= \chi_{k}+ h.o.t.$$
where 
$$\chi_{k}=a^2\partial_a+\frac{1}{m}ab\partial_b+\frac{1}{n}xy\partial_x+y^2\partial_y$$
and h.o.t. are terms of weight $> k$. Let $P_{k'}$ be the polynomial of lowest weight $k'>k$ in the defining equation of $S$, that is,
$y=a+b^mx^n+P_{k'}+R_{k'}$, and assume that the jet $F_{k+k'}$ is in normal form.
Then 
$$\chi= \chi_{k} + \chi_{k'} + h.o.t.$$ 
Indeed, if there was an integer $h\in (k,k')$ such that $\chi_h\neq 0$, then the terms of weight $k+h$ in $\chi(S)|_S$ would be given
by $\chi_h(y-a-b^mx^n)$. However, the latter expression is not zero because $\chi_h$ is not an automorphism of the model. So $\chi_h=0$.

Therefore, the terms of weight $k+k'$ are the lowest order terms in
$$\chi (y-a-b^mx^n-f(a,b,x))|_{y=a+b^mx^n+f(a,b,x)}$$
and they equal
$$ \left(\chi_{k'} (y-a-b^mx^n) -\chi_k P_{k'} \right)|_{y=a+b^mx^n{+f(a,b,x)}}$$
Direct computation, using the definition of normal form shows that $\chi_{k'} (y-a-b^mx^n)$ consists of non-normal terms  and $\chi_k P_{k'}$ consists of normal terms. This is only possible if $\chi_{k'}=0$ and $P_{k'}=0$, hence only if $S$ is the model.

If $\chi$ is a deformation of a linear vector field 
$$\chi_0= nb\partial_b- mx\partial_x$$
it follows that $\chi_0$ itself preserves $S$. Since, by the previous argument, $\chi$ is completely determined by $\chi_0$ we have $\chi=\chi_0$. 
Consequently, all terms in the Taylor expansion of $f$ have to be of the form $\# (b^mx^n)^r$.
\end{proof}

\end{document}